\long\def\symbolfootnote[#1]#2{\begingroup\def\thefootnote{\fnsymbol{footnote}}
\footnote[#1]{#2}\endgroup}

\documentclass{amsart}
\usepackage{amsfonts}
\usepackage{amsmath}

\newtheorem{theorem}{Theorem}[section]

\theoremstyle{remark}

\theoremstyle{definition}

\theoremstyle{proposition}

\numberwithin{equation}{section}

\begin{document}
\title{Topology of K\"{a}hler Ricci solitons}
\author{Ovidiu Munteanu and Jiaping Wang}
\maketitle

\begin{abstract}
We study the issue of connectedness at infinity of gradient K\"{a}hler Ricci
solitons. For shrinking K\"{a}hler Ricci solitons, we show they must be
connected at infinity. We also show the same holds true for expanding K\"{a}%
hler Ricci solitons with proper potential functions. As a separate issue, we
obtain a sharp pointwise lower bound for the weight function in terms of the
lower bound of the associated Bakry-\'Emery curvature.
\end{abstract}

\symbolfootnote[0]{The first author partially supported by NSF grant No.
DMS-1262140 and the second author by NSF grant No. DMS-1105799}

For a smooth function $f$ on a Riemannian manifold $\left( M,g\right) ,$ the
associated Bakry-\'{E}mery curvature, denoted by $\mathrm{Ric}_{f},$ is
defined by

\begin{equation*}
\mathrm{Ric}_{f}=\mathrm{Ric}+\mathrm{Hess}\left( f\right) ,
\end{equation*}%
where $\mathrm{Ric}$ is the Ricci curvature of $M$ and $\mathrm{Hess}\left(
f\right) $ the hessian of $f.$ This curvature notion has received much
attention recently. One of the reasons is that it admits a
generalization to more general metric spaces. Indeed, if one normalizes the
weighted measure $e^{-f}\,dv$ to be a probability measure on $M,$ where $dv$
is the volume form of $M,$ then Lott and Villani \cite{LV}, and
independently, Sturm \cite{S1, S2} have interpreted the lower bound of $%
\mathrm{Ric}_{f}$ as a convexity measurement for the Nash entropy on the
space of probability measures on $M$ equipped with the Wasserstein metric.
They further exploited this interpretation to define a notion of Ricci
curvature for general metric measure spaces. From the geometry point of
view, manifolds with constant Bakry-\'{E}mery curvature are of great
interest in the study of the Ricci flows. Recall a gradient Ricci soliton is
a manifold $\left( M,g\right) $ such that there exists a function $f\in
C^{\infty }\left( M\right) $ satisfying%
\begin{equation*}
\mathrm{Ric}_{f}=\lambda g
\end{equation*}%
for some constant $\lambda \in \mathbb{R}.$ Solitons are classified as
shrinking, steady or expanding, according to $\lambda >0,\ \lambda =0$ or $%
\lambda <0,$ respectively. The function $f$ is called the potential.
Customarily, the constant $\lambda $ is assumed to be $1/2,$ $0,$ or $-1/2$
by scaling. Obviously, Ricci solitons are natural generalizations of
Einstein manifolds. More significantly, they are self similar solutions to
the Ricci flows, and play a crucial role in the study of singularities of
the flows \cite{CLN}.

In this paper, we continue our study of the geometry and topology of
manifolds with Bakry-\'Emery curvature bounded from below, paying particular
attention to the gradient Ricci solitons. In our previous work \cite{MW}, we
have proved that a nontrivial steady gradient Ricci soliton must be
connected at infinity. While this topological information is of interest
itself, it also tells that one can not perturb the connected sum of any two
such solitons into a new one. In \cite{MW1}, a similar result was
established for expanding Ricci solitons with scalar curvature bounded below
by $S\geq -\frac{n}{2}+\frac{1}{2},$ where $n$ is the dimension of the
manifold. Note that by \cite{PRS} $S\geq -\frac{n}{2}$ for any expanding
Ricci soliton. Moreover, according to \cite{CW}, there exists nontrivial
expanding Ricci solitons with more than one end with scalar curvature
asymptotic to $-\frac{n}{2}$ at infinity.

The case of shrinking solitons is apparently more challenging. We speculate
that a shrinking Ricci soliton with more than one end must split as a direct
product of the real line with an Einstein manifold of positive scalar
curvature. This speculation is partially based on our work in \cite{MW3}.
Our first objective here is to confirm this speculation for gradient K\"{a}%
hler Ricci solitons.

\begin{theorem}
\label{shrinker}Let $\left( M,g,f\right) $ be a gradient K\"{a}hler
shrinking Ricci soliton. Then $\left( M,g\right) $ is connected at infinity i.e., it has only one end.
\end{theorem}

Let us point out that so far all known examples of nontrivial gradient
shrinking Ricci soliton are K\"{a}hler \cite{Ca1}. Previously, in \cite{MS},
Sesum and the first author have observed that K\"{a}hler shrinking Ricci
solitons have at most one non-parabolic end. Under the restrictive
assumption that the scalar curvature satisfies $\sup S<\frac{n}{2}-1,$ they
managed to rule out the existence of parabolic ends.

Our proof relies on the following result which is of interest itself. It
generalizes a result due to P. Li \cite{L}, which corresponds to the case $%
\varphi$ is constant.

\begin{theorem}
\label{vanishing}Let $\left( M,g\right) $ be a complete K\"{a}hler manifold.
Let $\varphi$ be a smooth real function on $M$ such that $J(\nabla \varphi)$
is a Killing vector field. Assume that there exists a constant $C>0$ such
that 
\begin{equation*}
\left\vert \nabla \varphi \right\vert \left( x\right) \leq C\,d\left(
x_{0},x\right) .
\end{equation*}%
Then any solution $u$ to $\Delta _{\varphi }u=0$ with 
\begin{equation*}
\int_{M}\left\vert \nabla u\right\vert ^{2}e^{-\varphi }<\infty
\end{equation*}
must be pluriharmonic. In particular, $u$ must be a constant if $\varphi$ is
proper.
\end{theorem}

Let us mention that we have studied other Liouville theorems under a similar
setting in \cite{MW2}.

We point out that the preceding result is applicable to K\"{a}hler Ricci
solitons. Indeed, in terms of local holomorphic coordinates, the K\"{a}hler
Ricci soliton equation can be rewritten into 
\begin{eqnarray*}
R_{\alpha \bar{\beta}}+f_{\alpha \bar{\beta}} &=&\lambda \,g_{\alpha \bar{%
\beta}} \\
f_{\alpha \beta } &=&0.
\end{eqnarray*}%
As pointed out in \cite{C}, $f_{\alpha \beta }=0$ is equivalent to the
vector filed $J\left( \nabla f\right) $ being a Killing vector field. It is
also well-known (see \cite{CZ}) that $\left\vert \nabla f\right\vert $ is at
most of linear growth for gradient Ricci solitons.

The general framework of our proof of Theorem \ref{shrinker} follows the
function theoretic approach developed by Li and Tam \cite{LT}, where they
used harmonic functions to detect the number of ends. Here, we instead work
with the weighted Laplacian

\begin{equation*}
\Delta _{\varphi }u:=\Delta u-\left\langle \nabla \varphi ,\nabla
u\right\rangle ,
\end{equation*}%
where $\varphi =-f.$ A technically challenging step is to show that all the
ends of a gradient shrinking soliton, not necessarily K\"{a}hler, are $%
\varphi $-nonparabolic, namely, there exists a nonconstant solution $u$ to $%
\Delta _{\varphi }u=0$ on each end $E$ with $u=1$ on $\partial E$ and $0\leq
u\leq 1$ on $E.$ Once this is done, one concludes there exists a nonconstant
solution $u$ to $\Delta _{\varphi }u=0$ with 
\begin{equation*}
\int_{M}\left\vert \nabla u\right\vert ^{2}e^{-\varphi }<\infty 
\end{equation*}%
if $M$ is not connected at infinity. Notice that $f,$ hence $\varphi ,$ is
proper on a shrinking gradient soliton according to \cite{CZ}. The proof of 
Theorem \ref{shrinker}  is
then finished by invoking Theorem \ref{vanishing}.

Our technique can be applied to obtain a corresponding result for expanding K%
\"{a}hler Ricci solitons.

\begin{theorem}
\label{expanding}Let $\left( M,g,f\right) $ be an expanding K\"{a}hler
gradient Ricci soliton. Assume that the potential $f$ is proper. Then $%
\left( M,g\right) $ has only one end.
\end{theorem}

Let us point out that it is necessary to impose the additional assumption
that $f$ is proper. This is because K\"{a}hler-Einstein manifolds, viewed as
trivial expanding Ricci solitons, may have more than one end.

As a separate issue, we also obtain a result concerning the weight function $%
f$ for a smooth metric measure space with its Bakry-\'{E}mery curvature
bounded below.

\begin{theorem}
\label{positive}Let $\left( M,g,e^{-f}dv\right) $ be a smooth metric measure
space with Bakry-\'{E}mery curvature bounded by $\mathrm{Ric}_{f}\geq \frac{%
\lambda }{2}g,$ where $\lambda \in \left\{ -1,0,1\right\} $. Then for any $%
x_{0}\in M$ fixed, there exist $a>0,$ depending only on dimension $n $ and $%
\sup_{B_{x_{0}}\left( 1\right) }\left\vert f\right\vert ,$ such that 
\begin{equation*}
f\left( x\right) \geq \frac{\lambda }{4}\,r^{2}\left( x\right) -a\,r\left(
x\right)
\end{equation*}%
for all $x\in M$ with $r\left( x\right)$ sufficiently large, where $r\left(
x\right) :=d\left( x_{0},x\right).$
\end{theorem}

Note that we make no assumption on $\left\vert \nabla f\right\vert $ here.
For gradient Ricci solitons, Theorem \ref{positive} was previously known.
Indeed, by the Bianchi identity, $\mathrm{Ric}_{f}=\lambda g$ implies (see 
\cite{H}) $S+\left\vert \nabla f\right\vert ^{2}=2\,\lambda f+C,$ where $C$
is some constant. Since $S\geq 0$, see \cite{Ch, Ca1}, in the case $\lambda \geq 0$ and 
$S\geq n\,\lambda $, see \cite{PRS}, if $\lambda <0,$ the aforementioned lower
bound on $f$ follows immediately when $\lambda \leq 0.$ In the case $\lambda
>0,$ Cao and Zhou \cite{CZ} have obtained the following estimate on
potential function $f.$ 
\begin{equation*}
\frac{\lambda }{2}\,\left( r\left( x\right) -c\right) ^{2}\leq f\left(
x\right) \leq \frac{\lambda }{2}\,\left( r\left( x\right) +c\right) ^{2}
\end{equation*}%
for all $x\in M$. These
arguments, however, rely on the information of $\left\vert \nabla
f\right\vert .$ Incidently, the results for gradient Ricci solitons imply
Theorem \ref{positive} is sharp.

\section{A vanishing theorem for $\protect\varphi$-harmonic functions}

In this section we prove Theorem \ref{vanishing}. Let us first fix some
notations. We let $\left(M,g\right) $ be a K\"{a}hler manifold and $\varphi$
a smooth real function on $M.$ Define unitary frame by 
\begin{eqnarray*}
\upsilon _{\alpha } &=&e_{\alpha }-iJe_{\alpha } \\
\upsilon _{\bar{\alpha}} &=&e_{\alpha }+iJe_{\alpha }
\end{eqnarray*}%
for $\alpha \in \left\{ 1,2,..,m\right\},$ where $\left\{ e_{\alpha
},Je_{\alpha }\right\} $ is an orthonormal frame for $g$ and $m$ the complex
dimension of $\left( M,g\right).$ With respect to such frames, we have 
\begin{eqnarray*}
\Delta u &=&u_{\alpha \bar{\alpha}} \\
\left\langle \nabla u,\nabla v\right\rangle &=&\frac{1}{2}\left( u_{\alpha
}v_{\bar{\alpha}}+u_{\bar{\alpha}}v_{\alpha }\right)
\end{eqnarray*}%
for any two functions $u,v\in C^{\infty }\left( M\right) $. Recall that the $%
\varphi -$Laplacian on functions is the operator $\Delta _{\varphi }:=\Delta
-\left\langle \nabla \varphi ,\nabla \right\rangle. $

We restate the theorem we are after.

\begin{theorem}
\label{vanishing1}For a K\"{a}hler manifold $\left( M,g\right) ,$ assume $%
\varphi :M\rightarrow \mathbb{R}$ satisfies $\varphi _{\alpha \beta }=0$ in
any unitary frame. Additionally, assume there exists a constant $C>0$ such
that 
\begin{equation*}
\left\vert \nabla \varphi \right\vert \left( x\right) \leq Cd\left(
x_{0},x\right) ,\text{ \ for any }x\in M.
\end{equation*}%
Then any function $u:M\rightarrow \mathbb{R}$ such that 
\begin{gather*}
\Delta _{\varphi }u=0 \\
\int_{M}\left\vert \nabla u\right\vert ^{2}e^{-\varphi }<\infty
\end{gather*}%
must be pluriharmonic. In particular, it is a constant if $\varphi$ is
proper on $M.$
\end{theorem}

\begin{proof}[Proof of Theorem \protect\ref{vanishing1}]
Our proof is an adaptation to manifolds with weights of a result in \cite{L}%
. Indeed, in \cite{L}, P. Li proved that a harmonic function with finite
Dirichlet integral is pluriharmonic.

Let us assume that there exists a function $u:M\rightarrow \mathbb{R}$ such
that 
\begin{gather}
\ \Delta _{\varphi }u=0  \label{v1} \\
\int_{M}\left\vert \nabla u\right\vert ^{2}e^{-\varphi }<\infty  \notag
\end{gather}%
For a fixed $x_{0}\in M,$ we consider the cut-off function defined by 
\begin{equation*}
\phi \left( x\right) =\left\{ 
\begin{array}{c}
1 \\ 
\frac{2R-d\left( x_{0},x\right) }{R} \\ 
0%
\end{array}%
\right. 
\begin{array}{c}
\text{on\ \ }B_{x_{0}}\left( R\right) \\ 
\text{on\ \ }B_{x_{0}}\left( 2R\right) \backslash B_{x_{0}}\left( R\right)
\\ 
\text{on }M\backslash B_{x_{0}}\left( 2R\right)%
\end{array}%
\end{equation*}%
Integration by parts implies 
\begin{gather*}
\int_{M}\left\vert u_{\alpha \bar{\beta}}\right\vert ^{2}\phi
^{2}e^{-\varphi }=\int_{M}u_{\alpha \bar{\beta}}u_{\beta \bar{\alpha}}\phi
^{2}e^{-\varphi } \\
=-\int_{M}u_{\alpha \bar{\beta}\bar{\alpha}}u_{\beta }\phi ^{2}e^{-\varphi
}+\int_{M}u_{\alpha \bar{\beta}}u_{\beta }\varphi _{\bar{\alpha}}\phi
^{2}e^{-\varphi }-\int_{M}u_{\alpha \bar{\beta}}u_{\beta }\left( \phi
^{2}\right) _{\bar{\alpha}}e^{-\varphi }.
\end{gather*}%
This can be rewritten into 
\begin{gather}
\int_{M}\left\vert u_{\alpha \bar{\beta}}\right\vert ^{2}\phi
^{2}e^{-\varphi }=-\int_{M}\mathrm{Re}\left( u_{\alpha \bar{\beta}\bar{\alpha%
}}u_{\beta }\right) \phi ^{2}e^{-\varphi }  \label{v2} \\
+\int_{M}\mathrm{Re}\left( u_{\alpha \bar{\beta}}u_{\beta }\varphi _{\bar{%
\alpha}}\right) \phi ^{2}e^{-\varphi }-\int_{M}\mathrm{Re}\left( u_{\alpha 
\bar{\beta}}u_{\beta }\left( \phi ^{2}\right) _{\bar{\alpha}}\right)
e^{-\varphi },  \notag
\end{gather}%
where $\mathrm{Re}\left( z\right) :=\frac{z+\bar{z}}{2}$ denotes the real
part of the complex number $z$.

We now investigate each term in (\ref{v2}). First, by the Ricci identities
on K\"{a}hler manifolds, we have 
\begin{gather*}
-\int_{M}\mathrm{Re}\left( u_{\alpha \bar{\beta}\bar{\alpha}}u_{\beta
}\right) \phi ^{2}e^{-\varphi }=-\int_{M}\mathrm{Re}\left( u_{\alpha \bar{%
\alpha}\bar{\beta}}u_{\beta }\right) \phi ^{2}e^{-\varphi } \\
=-\int_{M}\mathrm{Re}\left( \left( \Delta u\right) _{\bar{\beta}}u_{\beta
}\right) \phi ^{2}e^{-\varphi }=-\int_{M}\left\langle \nabla \Delta u,\nabla
u\right\rangle \phi ^{2}e^{-\varphi } \\
=\int_{M}\left( \Delta u\right) \left( \Delta _{\varphi }u\right) \phi
^{2}e^{-\varphi }+\int_{M}\left( \Delta u\right) \left\langle \nabla
u,\nabla \phi ^{2}\right\rangle e^{-\varphi } \\
=\int_{M}\left( \Delta u\right) \left\langle \nabla u,\nabla \phi
^{2}\right\rangle e^{-\varphi }.
\end{gather*}%
We thus conclude by the Cauchy-Schwarz inequality that%
\begin{gather}
-\int_{M}\mathrm{Re}\left( u_{\alpha \bar{\beta}\bar{\alpha}}u_{\beta
}\right) \phi ^{2}e^{-\varphi }\leq \frac{1}{4m}\int_{M}\left( \Delta
u\right) ^{2}\phi ^{2}e^{-\varphi }  \label{v3} \\
+4m\int_{M}\left\vert \nabla u\right\vert ^{2}\left\vert \nabla \phi
\right\vert ^{2}e^{-\varphi }  \notag \\
\leq \frac{1}{4}\int_{M}\left\vert u_{\alpha \bar{\beta}}\right\vert
^{2}\phi ^{2}e^{-\varphi }+4m\int_{M}\left\vert \nabla u\right\vert
^{2}\left\vert \nabla \phi \right\vert ^{2}e^{-\varphi }.  \notag
\end{gather}%
Secondly, integration by parts gives 
\begin{gather*}
\int_{M}\mathrm{Re}\left( u_{\alpha \bar{\beta}}u_{\beta }\varphi _{\bar{%
\alpha}}\right) \phi ^{2}e^{-\varphi }=-\int_{M}\mathrm{Re}\left( u_{\alpha
}u_{\beta \bar{\beta}}\varphi _{\bar{\alpha}}\right) \phi ^{2}e^{-\varphi }
\\
-\int_{M}\mathrm{Re}\left( u_{\alpha }u_{\beta }\varphi _{\bar{\alpha}\bar{%
\beta}}\right) \phi ^{2}e^{-\varphi }+\int_{M}\mathrm{Re}\left( u_{\alpha
}u_{\beta }\varphi _{\bar{\alpha}}\varphi _{\bar{\beta}}\right) \phi
^{2}e^{-\varphi } \\
-\int_{M}\mathrm{Re}\left( u_{\alpha }u_{\beta }\varphi _{\bar{\alpha}%
}\left( \phi ^{2}\right) _{\bar{\beta}}\right) e^{-\varphi },
\end{gather*}%
where we have invoked the assumption 
\begin{equation*}
\varphi _{\alpha \beta }=\varphi _{\bar{\alpha}\bar{\beta}}=0.
\end{equation*}%
Therefore, 
\begin{gather}
\int_{M}\mathrm{Re}\left( u_{\alpha \bar{\beta}}u_{\beta }\varphi _{\bar{%
\alpha}}\right) \phi ^{2}e^{-\varphi }=-\int_{M}\left( \Delta u\right) 
\mathrm{Re}\left( u_{\alpha }\varphi _{\bar{\alpha}}\right) \phi
^{2}e^{-\varphi }  \label{v4} \\
+\int_{M}\mathrm{Re}\left( u_{\alpha }u_{\beta }\varphi _{\bar{\alpha}%
}\varphi _{\bar{\beta}}\right) \phi ^{2}e^{-\varphi }-\int_{M}\mathrm{Re}%
\left( u_{\alpha }u_{\beta }\varphi _{\bar{\alpha}}\left( \phi ^{2}\right) _{%
\bar{\beta}}\right) e^{-\varphi }  \notag \\
=-\int_{M}\left\langle \nabla u,\nabla \varphi \right\rangle ^{2}\phi
^{2}e^{-\varphi }+\int_{M}\mathrm{Re}\left( u_{\alpha }u_{\beta }\varphi _{%
\bar{\alpha}}\varphi _{\bar{\beta}}\right) \phi ^{2}e^{-\varphi }  \notag \\
-\int_{M}\mathrm{Re}\left( u_{\alpha }u_{\beta }\varphi _{\bar{\alpha}%
}\left( \phi ^{2}\right) _{\bar{\beta}}\right) e^{-\varphi }.  \notag
\end{gather}%
Notice, however, that 
\begin{equation}
\mathrm{Re}\left( u_{\alpha }u_{\beta }\varphi _{\bar{\alpha}}\varphi _{\bar{%
\beta}}\right) \leq \left\langle \nabla u,\nabla \varphi \right\rangle ^{2}.
\label{*}
\end{equation}%
Plugging into (\ref{v4}), we conclude 
\begin{equation}
\int_{M}\mathrm{Re}\left( u_{\alpha \bar{\beta}}u_{\beta }\varphi _{\bar{%
\alpha}}\right) \phi ^{2}e^{-\varphi }\leq 2\int_{M}\left\vert \nabla
\varphi \right\vert \left\vert \nabla u\right\vert ^{2}\phi \left\vert
\nabla \phi \right\vert e^{-\varphi }.  \label{v5}
\end{equation}%
Using (\ref{v3}) and (\ref{v5}) in (\ref{v2}), we find%
\begin{gather}
\int_{M}\left\vert u_{\alpha \bar{\beta}}\right\vert ^{2}\phi
^{2}e^{-\varphi }\leq \frac{1}{2}\int_{M}\left\vert u_{\alpha \bar{\beta}%
}\right\vert ^{2}\phi ^{2}e^{-\varphi }  \label{v6} \\
+8m\int_{M}\left\vert \nabla u\right\vert ^{2}\left\vert \nabla \phi
\right\vert ^{2}e^{-\varphi }+2\int_{M}\left\vert \nabla \varphi \right\vert
\left\vert \nabla u\right\vert ^{2}\phi \left\vert \nabla \phi \right\vert
e^{-\varphi }.  \notag
\end{gather}%
Since $\int_{M}\left\vert \nabla u\right\vert ^{2}e^{-\varphi }<\infty ,$ it
is easy to see that as $R\rightarrow \infty ,$%
\begin{equation*}
\int_{M}\left\vert \nabla u\right\vert ^{2}\left\vert \nabla \phi
\right\vert ^{2}e^{-\varphi }\rightarrow 0.
\end{equation*}%
Furthermore, by the assumption that $\left\vert \nabla \varphi \right\vert
\leq Cd\left( x_{0},x\right),$ we have $\left\vert \nabla \varphi
\right\vert \left\vert \nabla \phi \right\vert \leq C$ on $M$. Consequently, 
\begin{equation*}
\int_{M}\left\vert \nabla \varphi \right\vert \left\vert \nabla u\right\vert
^{2}\phi \left\vert \nabla \phi \right\vert e^{-\varphi }\rightarrow 0.
\end{equation*}
We can now conclude that $u_{\alpha \bar{\beta}}=0$ or $u$ is pluriharmonic
by letting $R\rightarrow \infty $ in (\ref{v6}).

To show $u$ is constant in the case $\varphi$ is proper, we first note that $%
\left\langle \nabla u,\nabla \varphi \right\rangle =0$ as $u$ is both $%
\varphi$-harmonic and harmonic. Denote 
\begin{equation*}
D\left( t\right) :=\left\{ x:\left\vert \varphi \right\vert \left( x\right)
\leq t\right\} ,
\end{equation*}%
which is compact as $\varphi $ is proper. Now 
\begin{eqnarray*}
\int_{D\left( t\right) }\left\vert \nabla u\right\vert ^{2} &=&\frac{1}{2}%
\int_{D\left( t\right) }\Delta u^{2}=\frac{1}{2}\int_{\partial D\left(
t\right) }\frac{\partial u^{2}}{\partial \nu } \\
&=&\int_{\partial D\left( t\right) }u\frac{\left\langle \nabla u,\nabla
\varphi \right\rangle }{\left\vert \nabla \varphi \right\vert } \\
&=&0.
\end{eqnarray*}%
Since this is true for any $t,$ it follows that $\left\vert \nabla
u\right\vert =0$ on $M$ or $u$ is a constant. Theorem \ref{vanishing1} is
proved.
\end{proof}

\section{Ends of shrinking solitons}

In this section, we assume $\left( M,g,f\right) $ is a K\"{a}hler shrinking
Ricci soliton. So, in terms of unitary frames, we have 
\begin{eqnarray}
R_{\alpha \bar{\beta}}+f_{\alpha \bar{\beta}} &=&g_{\alpha \bar{\beta}}
\label{srs} \\
f_{\alpha \beta } &=&0.  \notag
\end{eqnarray}%
For a shrinking Ricci soliton, the potential $f$ satisfies \cite{CZ}%
\begin{equation*}
\frac{1}{4}r^{2}\left( x\right) -cr\left( x\right) \leq f\left( x\right)
\leq \frac{1}{4}r^{2}\left( x\right) +cr\left( x\right)
\end{equation*}%
for all $r\left( x\right) :=d\left( x_{0},x\right)$ sufficiently large.
Using the Bianchi identities, Hamilton \cite{H} proved 
\begin{equation*}
S+\left\vert \nabla f\right\vert ^{2}=f,
\end{equation*}%
where $S$ denotes the scalar curvature of $\left( M,g\right) $. Since $S\geq
0$ according to a result by Chen \cite{Ch}, we have 
\begin{equation*}
\left\vert \nabla f\right\vert ^{2}\leq f.
\end{equation*}%
We gather these properties here for future reference.%
\begin{gather}
\frac{1}{4}r^{2}\left( x\right) -cr\left( x\right) \leq f\left( x\right)
\leq \frac{1}{4}r^{2}\left( x\right) +cr\left( x\right)  \label{f} \\
\left\vert \nabla f\right\vert ^{2}\leq f  \notag \\
S\geq 0.  \notag
\end{gather}

We now prove the following theorem.

\begin{theorem}
\label{shrinker1}Let $\left( M,g,f\right) $ be a gradient K\"{a}hler
shrinking Ricci soliton. Then $\left( M,g\right) $ has only one end.
\end{theorem}

\begin{proof}[Proof of Theorem \protect\ref{shrinker1}]
While with respect to the weight function $f,$ the smooth metric space $%
\left( M,g,e^{-f}dv\right)$ has constant positive Bakry-\'{E}mery curvature $%
\mathrm{Ric}_{f}=\frac{1}{2}g,$ we will consider a different weight instead.
Let us set throughout this section 
\begin{equation}
\varphi :=-af,  \label{p}
\end{equation}%
where $a>0$ is a constant. We will be interested in the properties of $%
\left( M,g,e^{-\varphi }dv\right) .$ Clearly, the Bakry-\'{E}mery curvature
associated to this smooth metric measure space is given by 
\begin{eqnarray*}
\mathrm{Ric}_{\varphi } &=&\mathrm{Ric}+\mathrm{Hess}\left( \varphi \right)
\\
&=&\frac{1}{2}g-\left( a+1\right) \mathrm{Hess}\left( \varphi \right) .
\end{eqnarray*}%
The important feature is that $\left( M,g,e^{-\varphi}dv\right) $ has only $%
\varphi -$nonparabolic ends, a fact we will demonstrate below. Recall that
an end $E$ of $\left( M,g,e^{-\varphi }dv\right) $ is said to be $\varphi -$%
nonparabolic if there exists a positive Green's function for the weighted
Laplacian 
\begin{equation*}
\Delta _{\varphi }u:=\Delta u-\left\langle \nabla \varphi ,\nabla
u\right\rangle
\end{equation*}%
satisfying the Neumann boundary conditions on $\partial E.$ Otherwise, it is
called $\varphi -$parabolic. We refer to \cite{Li} for some important facts
concerning parabolicity. Although \cite{Li} studies the usual Laplacian $%
\Delta ,$ the results there extend without much effort to the weighted
Laplacian case.

We now show that $M$ does not admit any $\varphi -$parabolic ends. This fact
is true without assuming $\left( M,g\right) $ is K\"{a}hler. Suppose $E$ is
a $\varphi -$parabolic end of $M.$ A convenient way to characterize a $%
\varphi -$parabolic end, due to Nakai \cite{N}, is by the existence of a
proper $\varphi -$harmonic function $h$ on the end. So we have a function $%
h\geq 1$ defined on $E$ such that 
\begin{gather}
\lim_{x\rightarrow E\left( \infty \right) }h\left( x\right) =\infty \text{ \
and }h=1\text{ on }\partial E,  \label{s1} \\
\Delta _{\varphi }h=\Delta h+a\left\langle \nabla h,\nabla f\right\rangle =0.
\notag
\end{gather}%
Our goal is to show that (\ref{s1}) leads to a contradiction, which implies
that all ends of $\left( M,g\right) $ are $\varphi -$nonparabolic.

Since this part of the proof does not make use of the fact of $\left(
M,g\right) $ being K\"{a}hler, we will use moving frame notations with
indices from $i=1,..,n,$ where $n=2m$ is the real dimension of $\left(
M,g\right).$ Hence, $\left( u_{ij}\right)$ is the real hessian of the
function $u$ and $\left\langle \nabla u,\nabla v\right\rangle =u_{i}v_{i}$
in any orthonormal frame.

For $t>1$ and $b>c>1,$ we denote 
\begin{eqnarray*}
l\left( t\right)  &:&=\left\{ x\in E:h\left( x\right) =t\right\}  \\
L\left( c,b\right)  &:&=\left\{ x\in E:c<h\left( x\right) <b\right\} .
\end{eqnarray*}%
Notice that by (\ref{s1}) we know $l\left( t\right) $ and $L\left(
c,b\right) $ are compact. By the Stokes formula we get 
\begin{eqnarray*}
0 &=&\int_{L\left( c,b\right) }\left( \Delta _{\varphi }h\right) e^{-\varphi
} \\
&=&\int_{l\left( b\right) }\frac{\partial h}{\partial \nu }e^{-\varphi
}-\int_{l\left( c\right) }\frac{\partial h}{\partial \nu }e^{-\varphi } \\
&=&\int_{l\left( b\right) }\left\vert \nabla h\right\vert e^{-\varphi
}-\int_{l\left( c\right) }\left\vert \nabla h\right\vert e^{-\varphi },
\end{eqnarray*}%
where we have used that $\frac{\partial }{\partial \nu }=\frac{\nabla h}{%
\left\vert \nabla h\right\vert }$ is the unit normal to the level set of $h.$
This shows that $\int_{l\left( t\right) }\left\vert \nabla h\right\vert
e^{-\varphi }$ is independent of $t\geq 1.$ Now we apply co-area formula to
get (cf. \cite{LT}) 
\begin{eqnarray*}
\int_{E}\left\vert \nabla \ln h\right\vert ^{2}e^{-\varphi }
&=&\int_{L\left( 1,\infty \right) }\left\vert \nabla \ln h\right\vert
^{2}e^{-\varphi } \\
&=&\int_{1}^{\infty }\left( \int_{l\left( t\right) }\frac{\left\vert \nabla
h\right\vert }{h^{2}}e^{-\varphi }\right) dt \\
&=&\left( \int_{1}^{\infty }\frac{1}{t^{2}}dt\right) \int_{l\left(
t_{0}\right) }\left\vert \nabla h\right\vert e^{-\varphi }<\infty .
\end{eqnarray*}%
In particular, for any $x\in E$ with $B_{x}\left( 1\right) \subset E,$ we
have by (\ref{f})%
\begin{equation}
\int_{B_{x}\left( 1\right) }\left\vert \nabla \ln h\right\vert ^{2}\leq e^{-%
\frac{a}{4}r^{2}\left( x\right) +cr\left( x\right) },  \label{s2}
\end{equation}%
where $r\left( x\right) :=d\left( x_{0},x\right) $. We now transform (\ref%
{s2}) into a pointwise estimate. This is somewhat technical as $\left\vert
\nabla \ln h\right\vert ^{2}$ does not satisfy a convenient differential
inequality. Let 
\begin{equation*}
v:=\ln h.
\end{equation*}%
Then 
\begin{equation*}
\frac{1}{2}\Delta \left\vert \nabla v\right\vert ^{2}=\left\vert \mathrm{Hess%
}\left( v\right) \right\vert ^{2}+\left\langle \nabla \Delta v,\nabla
v\right\rangle +\mathrm{Ric}\left( \nabla v,\nabla v\right) .
\end{equation*}%
Since by (\ref{s1}) 
\begin{eqnarray}
\Delta v &=&\frac{1}{h}\Delta h-\frac{1}{h^{2}}\left\vert \nabla
h\right\vert ^{2}  \label{a1} \\
&=&-a\left\langle \nabla v,\nabla f\right\rangle -\left\vert \nabla
v\right\vert ^{2},  \notag
\end{eqnarray}%
we conclude that 
\begin{eqnarray*}
\frac{1}{2}\Delta \left\vert \nabla v\right\vert ^{2} &=&\left\vert \mathrm{%
Hess}\left( v\right) \right\vert ^{2}-a\left\langle \nabla \left\langle
\nabla v,\nabla f\right\rangle ,\nabla v\right\rangle -\left\langle \nabla
\left\vert \nabla v\right\vert ^{2},\nabla v\right\rangle  \\
&&+\frac{1}{2}\left\vert \nabla v\right\vert ^{2}-\mathrm{Hess}\left(
f\right) \left( \nabla v,\nabla v\right)  \\
&\geq &\frac{1}{2}\left\vert \mathrm{Hess}\left( v\right) \right\vert
^{2}-a^{2}f\left\vert \nabla v\right\vert ^{2}-\left\langle \nabla
\left\vert \nabla v\right\vert ^{2},\nabla v\right\rangle  \\
&&-\left( a+1\right) \mathrm{Hess}\left( f\right) \left( \nabla v,\nabla
v\right) .
\end{eqnarray*}%
By the Cauchy-Schwarz inequality and (\ref{a1}), we obtain%
\begin{eqnarray*}
\left\vert \mathrm{Hess}\left( v\right) \right\vert ^{2} &\geq &\frac{1}{n}%
\left( \Delta v\right) ^{2} \\
&=&\frac{1}{n}\left( a\left\langle \nabla v,\nabla f\right\rangle
-\left\vert \nabla v\right\vert ^{2}\right) ^{2} \\
&\geq &\frac{1}{2n}\left\vert \nabla v\right\vert ^{4}-a^{2}f\left\vert
\nabla v\right\vert ^{2}.
\end{eqnarray*}%
For convenience, let us denote 
\begin{equation*}
\sigma :=\left\vert \nabla v\right\vert ^{2}.
\end{equation*}%
In conclusion, we have the following inequality%
\begin{equation}
\sigma ^{2}\leq 8na^{2}f\sigma +\,4n\left\langle \nabla \sigma ,\nabla
v\right\rangle +4n\left( a+1\right) f_{ij}v_{i}v_{j}+2n\Delta \sigma .
\label{s3}
\end{equation}%
Multiplying (\ref{s3}) by $\phi ^{2}\sigma ^{p-1}$ and integrating over $M,$
where $\phi $ is a cut-off function with support in $B_{x}\left( 1\right) ,$
we get 
\begin{gather}
\int_{M}\sigma ^{p+1}\ \phi ^{2}\leq 8na^{2}\int_{M}f\sigma ^{p}\phi
^{2}+4n\int_{M}\left\langle \nabla \sigma ,\nabla v\right\rangle \sigma
^{p-1}\phi ^{2}  \label{s4} \\
+4n\left( a+1\right) \int_{M}f_{ij}v_{i}v_{j}\sigma ^{p-1}\phi
^{2}+2n\int_{M}\sigma ^{p-1}\left( \Delta \sigma \right) \phi ^{2}.  \notag
\end{gather}%
Notice now that 
\begin{gather}
\int_{M}\left\langle \nabla \sigma ,\nabla v\right\rangle \sigma ^{p-1}\phi
^{2}=\frac{1}{p}\int_{M}\left\langle \nabla \sigma ^{p},\nabla
v\right\rangle \ \phi ^{2}  \label{s5} \\
=-\frac{1}{p}\int_{M}\sigma ^{p}\Delta v\ \phi ^{2}-\frac{1}{p}%
\int_{M}\sigma ^{p}\left\langle \nabla v,\nabla \phi ^{2}\right\rangle  
\notag \\
=\frac{1}{p}\int_{M}\sigma ^{p}\left( a\left\langle \nabla v,\nabla
f\right\rangle +\sigma \right) \ \phi ^{2}-\frac{1}{p}\int_{M}\sigma
^{p}\left\langle \nabla v,\nabla \phi ^{2}\right\rangle   \notag \\
\leq \frac{2}{p}\int_{M}\sigma ^{p+1}\phi ^{2}+\frac{2}{p}\int_{M}\left\vert
\nabla \phi \right\vert ^{2}\sigma ^{p}+\frac{2}{p}a^{2}\int_{M}f\sigma
^{p}\phi ^{2}.  \notag
\end{gather}%
Integrating by parts, we get 
\begin{eqnarray}
\int_{M}\sigma ^{p-1}\left( \Delta \sigma \right) \phi ^{2} &=&-\left(
p-1\right) \int_{M}\left\vert \nabla \sigma \right\vert ^{2}\sigma
^{p-2}\phi ^{2}-\int_{M}\sigma ^{p-1}\left\langle \nabla \sigma ,\nabla \phi
^{2}\right\rangle   \label{s6} \\
&\leq &-\left( p-2\right) \int_{M}\left\vert \nabla \sigma \right\vert
^{2}\sigma ^{p-2}\phi ^{2}+\int_{M}\sigma ^{p}\left\vert \nabla \phi
\right\vert ^{2}.  \notag
\end{eqnarray}%
Similarly, integration by parts gives%
\begin{gather}
\int_{M}f_{ij}v_{i}v_{j}\sigma ^{p-1}\phi
^{2}=-\int_{M}v_{ij}v_{j}f_{i}\sigma ^{p-1}\phi ^{2}-\int_{M}\left( \Delta
v\right) f_{i}v_{i}\sigma ^{p-1}\phi ^{2}  \label{s7} \\
-\left( p-1\right) \int_{M}f_{i}v_{i}v_{j}\sigma _{j}\sigma ^{p-2}\phi
^{2}-\int_{M}f_{i}v_{i}v_{j}\sigma ^{p-1}\left( \phi ^{2}\right) _{j}. 
\notag
\end{gather}%
We now study each term on the right hand side. 
\begin{eqnarray*}
-\int_{M}v_{ij}v_{j}f_{i}\sigma ^{p-1}\phi ^{2} &=&-\frac{1}{2}%
\int_{M}\left\langle \nabla \left\vert \nabla v\right\vert ^{2},\nabla
f\right\rangle \sigma ^{p-1}\phi ^{2} \\
&=&-\frac{1}{2p}\int_{M}\left\langle \nabla \sigma ^{p},\nabla
f\right\rangle \phi ^{2} \\
&=&\frac{1}{2p}\int_{M}\sigma ^{p}\left( \Delta f\right) \phi ^{2}+\frac{1}{%
2p}\int_{M}\sigma ^{p}\left\langle \nabla f,\nabla \phi ^{2}\right\rangle  \\
&\leq &\int_{M}f\sigma ^{p}\phi ^{2}+\int_{M}\sigma ^{p}\left\vert \nabla
\phi \right\vert ^{2}.
\end{eqnarray*}%
Also, 
\begin{eqnarray*}
-\int_{M}\left( \Delta v\right) f_{i}v_{i}\sigma ^{p-1}\phi ^{2}
&=&\int_{M}\left( a\left\langle \nabla v,\nabla f\right\rangle +\left\vert
\nabla v\right\vert ^{2}\right) \left\langle \nabla f,\nabla v\right\rangle
\sigma ^{p-1}\phi ^{2} \\
&\leq &npa\int_{M}f\sigma ^{p}\phi ^{2}+\frac{1}{npa}\int_{M}\sigma
^{p+1}\phi ^{2}.
\end{eqnarray*}%
The third term in the right hand side of (\ref{s7}) is estimated by 
\begin{gather*}
-\left( p-1\right) \int_{M}f_{i}v_{i}v_{j}\sigma _{j}\sigma ^{p-2}\phi
^{2}\leq p\int_{M}\left\vert \nabla f\right\vert \left\vert \nabla \sigma
\right\vert \sigma ^{p-1}\phi ^{2} \\
\leq \frac{p-2}{8na}\int_{M}\left\vert \nabla \sigma \right\vert ^{2}\sigma
^{p-2}\phi ^{2}+4npa\int_{M}f\sigma ^{p}\phi ^{2}.
\end{gather*}%
The last term in (\ref{s7}) is 
\begin{eqnarray*}
-\int_{M}f_{i}v_{i}v_{j}\sigma ^{p-1}\left( \phi ^{2}\right) _{j} &\leq
&2\int_{M}\left\vert \nabla f\right\vert \sigma ^{p}\phi \left\vert \nabla
\phi \right\vert  \\
&\leq &\int_{M}f\sigma ^{p}\phi ^{2}+\int_{M}\sigma ^{p}\left\vert \nabla
\phi \right\vert ^{2}.
\end{eqnarray*}%
Plugging all these estimates in (\ref{s7}), we conclude 
\begin{gather}
4n\left( a+1\right) \int_{M}f_{ij}v_{i}v_{j}\sigma ^{p-1}\phi ^{2}\leq
a^{3}p\int_{M}f\sigma ^{p}\phi ^{2}+\frac{4}{p}\int_{M}\sigma ^{p+1}\phi ^{2}
\label{s8} \\
+\frac{p-2}{2}\int_{M}\left\vert \nabla \sigma \right\vert ^{2}\sigma
^{p-2}\phi ^{2}+a^{2}\int_{M}\sigma ^{p}\left\vert \nabla \phi \right\vert
^{2}.  \notag
\end{gather}%
Using (\ref{s5}), (\ref{s6}) and (\ref{s8}), we get from (\ref{s4}) that for 
$p$ large enough and depending on $n,$ 
\begin{equation}
\ \left( p-2\right) \int_{M}\left\vert \nabla \sigma \right\vert ^{2}\sigma
^{p-2}\phi ^{2}\leq a^{3}p\int_{M}f\sigma ^{p}\phi ^{2}\
+a^{2}\int_{M}\sigma ^{p}\left\vert \nabla \phi \right\vert ^{2}.  \label{s9}
\end{equation}%
Note that (\ref{s9}) is true for any $\phi $ with support in $B_{x}\left(
1\right) $. We now recall the Sobolev inequality established in \cite{MW},
which says that there exist constants $\mu >1,$ $c_{1}$ and $c_{2},$ all
depending only on $n=2m$ such that 
\begin{equation*}
\left( \int_{B_{x}\left( 1\right) }\psi ^{2\mu }\right) ^{\frac{1}{\mu }%
}\leq \frac{c_{1}e^{c_{2}A}}{\mathrm{Vol}\left( B_{x}\left( 1\right) \right)
^{2\left( 1-\frac{1}{\mu }\right) }}\int_{B_{x}\left( 1\right) }\left\vert
\nabla \psi \right\vert ^{2}+\frac{c_{1}}{\mathrm{Vol}\left( B_{x}\left(
1\right) \right) }\int_{B_{x}\left( 1\right) }\psi ^{2}\ 
\end{equation*}%
for $\ \psi \in C_{0}^{\infty }\left( B_{x}\left( 1\right) \right) ,$
where 
\begin{equation*}
A:=\mathrm{osc}_{B_{x}\left( 3\right) }\left\vert f\right\vert ,
\end{equation*}%
the oscillation of $f$ over the geodesic ball $B_{x}(3).$ We also note that
there exists a constant $c(n)>0$ depending only on dimension $n$ \cite{MW3}
such that 
\begin{equation*}
\mathrm{Vol}\left( B_{x}\left( 1\right) \right) \geq \mathrm{Vol}\left(
B_{x_{0}}\left( 1\right) \right) \,e^{-C(n)\,d\left( x_{0},x\right) }.
\end{equation*}%
Therefore, using (\ref{f}), we arrive at a Sobolev inequality of the form 
\begin{equation}
\left( \int_{B_{x}\left( 1\right) }\psi ^{2\mu }\right) ^{\frac{1}{\mu }%
}\leq \ Ce^{c\left( n\right) r\left( x\right) }\int_{B_{x}\left( 1\right)
}\left\vert \nabla \psi \right\vert ^{2}+C\int_{B_{x}\left( 1\right) }\psi
^{2}  \label{S}
\end{equation}%
for all $\psi \in C_{0}^{\infty }\left( B_{x}\left( 1\right) \right) ,$
where $r\left( x\right) :=d\left( x_{0},x\right) .$ Using (\ref{s9}) and (%
\ref{S}), by the standard DeGiorgi-Nash-Moser iteration, we obtain for any $%
0\leq \theta ,\rho <1,$%
\begin{equation*}
\sup_{B_{x}\left( \theta \rho \right) }\sigma \leq C\left( 1+\left( 1-\theta
\right) ^{-2}\rho ^{-2}\right) ^{\frac{\mu }{\mu -1}\frac{1}{p}}e^{c\left(
n\right) r\left( x\right) }\left( \frac{1}{\mathrm{Vol}\left(
B_{x_{0}}\left( \rho \right) \right) \,}\int_{B_{x}\left( \rho \right)
}\sigma ^{p}\right) ^{\frac{1}{p}}
\end{equation*}%
for some $p$ depending only on dimension $n$ and given by (\ref{s9}). Here
the constant $C$ depends on $a,$ but it is independent of $r\left( x\right) $%
.

Now a standard argument (see \cite{Li}) implies 
\begin{equation}
\sigma \left( x\right) \leq Ce^{c\left( n\right) r\left( x\right)
}\int_{B_{x}\left( 1\right) }\sigma .  \label{s10}
\end{equation}%
Combining (\ref{s10}) with (\ref{s2}), one concludes 
\begin{equation}
\left\vert \nabla \ln h\left( x\right) \right\vert \leq \sqrt{\sigma }\left(
x\right) \leq c\,e^{-\frac{a}{16}\,r^{2}\left( x\right) }.  \label{s11}
\end{equation}%
Integrating (\ref{s11}) along minimizing geodesics we immediately see that $%
h $ must be bounded on the end $E,$ which contradicts with (\ref{s1}). This
proves all ends of $\left( M,g\right) $ are $\varphi -$nonparabolic.

To finish the proof of the Theorem, we assume to the contrary that $\left(
M,g\right) $ has more than one end. Since each end must be $\varphi -$%
nonparabolic, the construction of Li and Tam \cite{LT, Li} implies that
there exists a nonconstant $\varphi -$harmonic function $u:M\rightarrow 
\mathbb{R}$ such that 
\begin{eqnarray*}
\Delta _{\varphi }u &=&0 \\
\int_{M}\left\vert \nabla u\right\vert ^{2}e^{-\varphi } &<&\infty .
\end{eqnarray*}%
But $M$ as a K\"{a}hler shrinking Ricci soliton satisfies all the
assumptions of Theorem \ref{vanishing} with $\varphi=-a\,f.$ Indeed, (\ref%
{srs}) implies that $\varphi _{\alpha \beta}=0.$ Also (\ref{f}) shows $%
\varphi $ is proper and $\left\vert \nabla \varphi \right\vert $ grows at
most linearly in the distance function. Therefore, we conclude that $u$ must
be a constant. This contradiction shows that $M$ has only one end. Theorem %
\ref{shrinker1} is proved.
\end{proof}

\section{Ends of expanding solitons}

In this section, we show an expanding K\"{a}hler Ricci soliton must be
connected at infinity if its potential function is proper. First, let us
recall some useful information about expanding Ricci solitons. An expanding
gradient Ricci soliton is a manifold $\left( M,g\right) $ for which there
exists a smooth potential $f$ with the property that 
\begin{equation*}
\mathrm{Ric}+\mathrm{Hess}\left( f\right) =-\frac{1}{2}g.
\end{equation*}%
Since $\left( M,g\right) $ is K\"{a}hler, in terms of unitary frames, we
have 
\begin{eqnarray}
R_{\alpha \bar{\beta}}+f_{\alpha \bar{\beta}} &=&-g_{\alpha \bar{\beta}}
\label{ers} \\
f_{\alpha \beta } &=&0.  \notag
\end{eqnarray}%
Hamilton's identity now reads as 
\begin{equation*}
S+\left\vert \nabla f\right\vert ^{2}=-f.
\end{equation*}%
It is known \cite{PRS} that $S\geq -\frac{n}{2}$ with equality holding if
and only if $\left( M,g\right) $ is an Einstein manifold. So one has $%
\left\vert \nabla f\right\vert ^{2}\leq \left( -f\right) +\frac{n}{2}$ and 
\begin{equation*}
-f\left( x\right) \leq \frac{1}{4}r^{2}\left( x\right) +cr\left( x\right) .
\end{equation*}%
Let us summarize these facts as 
\begin{eqnarray}
-f\left( x\right)  &\leq &\frac{1}{4}r^{2}\left( x\right) +cr\left( x\right) 
\label{scal} \\
S &>&-\frac{n}{2}  \notag \\
\left\vert \nabla f\right\vert ^{2} &<&\left( -f\right) +\frac{n}{2},  \notag
\end{eqnarray}%
for any non-trivial expanding gradient Ricci soliton.

We now prove Theorem \ref{expanding}.

\begin{theorem}
\label{expanding1}Let $\left( M,g,f\right) $ be an expanding K\"{a}hler
gradient Ricci soliton. Assume that the potential $f$ is proper. Then $%
\left( M,g\right) $ has only one end.
\end{theorem}

\begin{proof}[Proof of Theorem \protect\ref{expanding1}]
Since $-f>-\frac{n}{2},$ the assumption that $f$ is proper implies that 
\begin{equation*}
\lim_{x\rightarrow \infty }\left( -f\right) \left( x\right) =\infty .
\end{equation*}%
We will use several ideas from \cite{MW1} with some improvements. It was
proved in \cite{MW1} that the following weighted Poincar\'{e} inequality
holds on $M.$%
\begin{equation}
\int_{M}\left( S+\frac{n}{2}\right) \phi ^{2}e^{-f}\leq \int_{M}\left\vert
\nabla \phi \right\vert ^{2}e^{-f}  \label{wpi}
\end{equation}%
for all $\phi \in C_{0}^{\infty }\left( M\right) .$ For our purpose,
however, a different inequality will be more useful. We compute 
\begin{eqnarray*}
\Delta _{f}e^{af} &=&\left( a\Delta _{f}\left( f\right) +a^{2}\left\vert
\nabla f\right\vert ^{2}\right) e^{af} \\
&=&-\left( a\left( \frac{n}{2}+S\right) +\left( a-a^{2}\right) \left\vert
\nabla f\right\vert ^{2}\right) e^{af}.
\end{eqnarray*}%
Choosing $a=1$ yields (\ref{wpi}). We take $a=\frac{1}{2}$ instead and get 
\begin{equation*}
\int_{M}\sigma \phi ^{2}e^{-f}\leq \int_{M}\left\vert \nabla \phi
\right\vert ^{2}e^{-f}
\end{equation*}%
for all $\phi \in C_{0}^{\infty }\left( M\right) ,$ where (see (\ref{scal})) 
\begin{equation*}
\sigma :=\frac{1}{4}\left( -f+\frac{n}{2}\right) >0\text{. }
\end{equation*}%
Since $f$ is proper, it means that there exists a compact set $K\subset M$
such that $-f\geq 3$ on $M\backslash K.$ In particular, it follows that 
\begin{equation}
\sigma \geq 1\text{ \ \ on \ }M\backslash K.  \label{e}
\end{equation}%
In view of (\ref{e}), the argument in \cite{MW1} now shows that all ends of $%
M$ are necessarily $f-$nonparabolic. To finish, assume by contradiction that 
$M$ has more than one end. Again, by the construction of Li and Tam \cite{LT}%
, one finds a nontrivial $f-$harmonic function $u$ on $M$ with finite total
energy $\int_{M}\left\vert \nabla u\right\vert ^{2}e^{-f}<\infty .$ On the
other hand, by (\ref{ers}) and (\ref{scal}), Theorem \ref{vanishing} is
applicable. So we conclude $u$ must be a constant. This contradiction proves
the theorem.
\end{proof}

\section{Weight function estimate}

In this section we prove Theorem \ref{positive}. For convenience, we only
provide the details for the case $\lambda =1$ as the other two cases $%
\lambda =0$ and $\lambda =-1$ follow similarly.

\begin{theorem}
\label{positive1}Let $\left( M,g,e^{-f}dv\right) $ be a smooth metric
measure space with Bakry-\'{E}mery curvature bounded by $\mathrm{Ric}%
_{f}\geq \frac{1}{2}g$. Then for any fixed $x_{0}\in M,$ there exists $a>0$
depending only on dimension $n$ and $\sup_{B_{x_{0}}\left( 1\right)
}\left\vert f\right\vert $ such that 
\begin{equation}
f\left( x\right) \geq \frac{1}{4}\,r^{2}\left( x\right) -a\,r\left( x\right)
\label{lb}
\end{equation}%
for all $x\in M$ with $r\left( x\right)$ sufficiently large, where $r\left(
x\right) :=d\left( x_{0},x\right).$
\end{theorem}

\begin{proof}[Proof of Theorem \protect\ref{positive1}]
The strategy is to use an improved Laplacian comparison theorem established
in \cite{MW3}. We first discuss its derivation for completeness. Our
derivation here is in fact slightly more direct than that in \cite{MW3}. For
a fixed point $p\in M,$ we denote the volume form in geodesic coordinates by 
\begin{equation*}
dV|_{\exp _{p}\left( r\xi \right) }=J\left( p,r,\xi \right) drd\xi
\end{equation*}%
for $r>0$ and $\xi \in S_{p}M,$ the unit tangent sphere at $p.$ For $x\in M$
a point outside the cut locus of $p$ with $x=\exp _{p}$ $\left( r\xi \right)
,$ we have%
\begin{equation*}
\Delta d\left( p,x\right) =\frac{d}{dr}\ln J\left( p,r,\xi \right) .
\end{equation*}%
We shall omit the dependency of these quantities on $p$ and $\xi $ from now
on. Along a minimizing geodesic $\gamma $ starting from $p$, by the Bochner
formula, we have 
\begin{equation}
m^{\prime }\left( r\right) +\frac{1}{n-1}m^{2}\left( r\right) +\mathrm{Ric}%
\left( \frac{\partial }{\partial r},\frac{\partial }{\partial r}\right) \leq
0,  \label{w1}
\end{equation}%
where the differentiation is with respect to the $r$ variable and $m\left(
r\right) :=\frac{d}{dr}\ln J\left( r\right) .$ Multiplying (\ref{w1}) by $r$
and integrating from $r=\varepsilon >0$ to $r=t>\varepsilon ,$ we get%
\begin{equation}
\int_{\varepsilon }^{t}m^{\prime }\left( r\right) rdr+\frac{1}{n-1}%
\int_{\varepsilon }^{t}m^{2}\left( r\right) rdr+\frac{1}{2}\int_{\varepsilon
}^{t}rdr\leq \int_{\varepsilon }^{t}f^{\prime \prime }\left( r\right) r\,dr,
\label{w2}
\end{equation}%
where we have used $\mathrm{Ric}_{f}\geq \frac{1}{2}.$ Integrating by parts
in (\ref{w2}), we arrive at%
\begin{gather*}
m\left( t\right) t-m\left( \varepsilon \right) \varepsilon +\frac{1}{n-1}%
\int_{\varepsilon }^{t}r\left( m\left( r\right) -\frac{n-1}{r}\right)
^{2}dr\leq \int_{\varepsilon }^{r}\left( \frac{n-1}{r}-m\left( r\right)
\right) dr \\
-\frac{1}{4}t^{2}+\frac{1}{4}\varepsilon ^{2}+\int_{\varepsilon
}^{t}f^{\prime \prime }\left( r\right) r\,dr.
\end{gather*}%
In particular, this yields the following%
\begin{equation*}
m\left( t\right) t+\ln \frac{J\left( t\right) }{t^{n-1}}\leq m\left(
\varepsilon \right) \varepsilon +\ln \frac{J\left( \varepsilon \right) }{%
\varepsilon ^{n-1}}-\frac{1}{4}t^{2}+\frac{1}{4}\varepsilon
^{2}+\int_{\varepsilon }^{t}f^{\prime \prime }\left( r\right) r\,dr.
\end{equation*}%
Now letting $\varepsilon \rightarrow 0,$ noting that $\varepsilon m\left(
\varepsilon \right) \rightarrow n-1$ and $\frac{J\left( \varepsilon \right) 
}{\varepsilon ^{n-1}}\rightarrow 1,$ we obtain 
\begin{equation*}
m\left( t\right) +\frac{1}{t}\ln \frac{J\left( t\right) }{t^{n-1}}\leq \frac{%
n-1}{t}-\frac{1}{4}t+\frac{1}{t}\int_{0}^{t}f^{\prime \prime }\left(
r\right) rdr.
\end{equation*}%
Integrating by parts on the last term, we get%
\begin{equation}
\Delta _{f}d\left( p,x\right) \leq \frac{n-1}{r}-\frac{1}{4}r-\frac{1}{r}\ln
\left( \frac{J\left( p,r,\xi \right) }{r^{n-1}}\right) -\frac{1}{r}\left(
f\left( x\right) -f\left( p\right) \right).  \label{w3}
\end{equation}

The proof of the theorem is by contradiction. So we fix a large enough
constant $a>0,$ to be determined explicitly later, and assume that for this
fixed $a$ there exists no $r_{0}$ such that (\ref{lb}) is true. This means
there exists a sequence $q_{k}\rightarrow \infty $ such that 
\begin{equation}
f\left( q_{k}\right) \leq \frac{1}{4}\,r^{2}\left( q_{k}\right) -a\,r\left(
q_{k}\right) ,  \label{w4}
\end{equation}%
where $r\left( x\right) :=d\left( x_{0},x\right) $. We now adapt to our
setting some ideas in \cite{MW3} for proving a splitting theorem for certain
smooth metric measure spaces containing a line.

More precisely, let $\gamma _{k}\left( t\right) $ be the minimizing geodesic
from $x_{0}$ to $q_{k},$ where 
\begin{equation*}
0\leq t\leq t_{k}:=d\left( x_{0},q_{k}\right) .
\end{equation*}%
For a fixed point $x\in M$, not in the cut locus of $q_{k}=\gamma _{k}\left(
t_{k}\right) $, let $\tau _{k}\left( s\right) $ be the unique minimizing
geodesic from $\tau _{k}\left( 0\right) =q_{k}$ to $\tau _{k}\left(
r_{k}\right) :=x$. By (\ref{w4}), we have 
\begin{equation}
f\left( q_{k}\right) \leq \frac{1}{4}t_{k}^{2}-at_{k} \text{ \ \ for all }k.
\label{w5}
\end{equation}%
Let 
\begin{equation*}
r_{k}:=d\left( \gamma \left( t_{k}\right) ,x\right) =d\left( q_{k},x\right) .
\end{equation*}%
According to (\ref{w3}), we have%
\begin{gather}
\Delta _{f}d\left( q_{k},x\right) +\frac{1}{r_{k}}\ln J\left(
q_{k},r_{k},\tau _{k}^{\prime }\left( 0\right) \right) \leq \frac{n-1}{r_{k}}%
\left( 1+\ln r_{k}\right) -\frac{1}{4}r_{k}  \label{w6} \\
+\frac{f\left( q_{k}\right) }{r_{k}}-\frac{f\left( x\right) }{r_{k}}.  \notag
\end{gather}%
We now claim that for any $\varepsilon >0$ and $k$ sufficiently large, 
\begin{equation}
\frac{n-1}{r_{k}}\left( 1+\ln r_{k}\right) -\frac{1}{4}r_{k}+\frac{f\left(
q_{k}\right) }{r_{k}}-\frac{f\left( x\right) }{r_{k}}\leq \frac{1}{2}\left(
t_{k}-r_{k}\right) -a+\varepsilon .  \label{w7}
\end{equation}%
To verify this, first note that by the triangle inequality, 
\begin{equation}
\left\vert t_{k}-r_{k}\right\vert \leq d\left( x_{0},x\right) .  \label{w8}
\end{equation}%
Since $x$ is fixed, for $k$ sufficiently large we have 
\begin{equation*}
\frac{n-1}{r_{k}}\left( 1+\ln r_{k}\right) -\frac{f\left( x\right) }{r_{k}}%
\leq \frac{\varepsilon }{2}.
\end{equation*}%
Using (\ref{w5}), we have,%
\begin{eqnarray*}
&&-\frac{1}{4}r_{k}+\frac{f\left( q_{k}\right) }{r_{k}}-\frac{1}{2}\left(
t_{k}-r_{k}\right) +a=\frac{1}{4r_{k}}\left( -r_{k}^{2}+4f\left(
q_{k}\right) -2r_{k}\left( t_{k}-r_{k}\right) +4ar_{k}\right) \\
&\leq &\frac{1}{4r_{k}}\left( -r_{k}^{2}+t_{k}^{2}-4at_{k}-2r_{k}\left(
t_{k}-r_{k}\right) +4ar_{k}\right) =\frac{1}{4r_{k}}\left( \left(
t_{k}-r_{k}\right) ^{2}-4a\left( t_{k}-r_{k}\right) \right) \\
&\leq &\frac{1}{4r_{k}}\left( d\left( x_{0},x\right) ^{2}+4ad\left(
x_{0},x\right) \right) ,
\end{eqnarray*}%
where we have used (\ref{w8}) in the last step. Hence, for $k$ sufficiently
large, 
\begin{equation*}
-\frac{1}{4}r_{k}+\frac{f\left( q_{k}\right) }{r_{k}}-\frac{1}{2}\left(
t_{k}-r_{k}\right) +a\leq \frac{\varepsilon }{2}.
\end{equation*}%
In conclusion, (\ref{w7}) holds true. It then follows from (\ref{w6}) that,
for $k$ sufficiently large, 
\begin{equation}
\Delta _{f}d\left( q_{k},x\right) +\frac{1}{r_{k}}\ln J\left(
q_{k},r_{k},\xi _{k}\right) \leq \frac{1}{2}\left( t_{k}-r_{k}\right)
-a+\varepsilon ,  \label{w9}
\end{equation}%
where $\xi _{k}:=\tau _{k}^{\prime }\left( 0\right) ,$ $t_{k}:=d\left(
x_{0},q_{k}\right) $, and $r_{k}:=d\left( q_{k},x\right) .$ Let us emphasize
that (\ref{w9}) holds for any $x$ not in the cut-locus of $q_{k}.$

For compact domain $\Omega \subset M,$ there exists a constant $c>0$
independent of $k$ so that 
\begin{equation*}
\Omega \subset B_{q_{k}}\left( t_{k}+c\right) \backslash B_{q_{k}}\left(
t_{k}-c\right) .
\end{equation*}%
Consequently, there exists a constant $c_{0}>0$ so that $r_{k}>t_{k}-c_{0}$
whenever $x\in \Omega .$ Also, note that the function $h\left( J\right)
:=J\ln J$ is bounded below by $-\frac{1}{e}.$ For nonnegative smooth
function $\phi $ with support in $\Omega ,$ multiplying (\ref{w9}) by $\phi $
and integrating over $\Omega ,$ we conclude that 
\begin{gather}
\int_{M}\left( \Delta _{f}d\left( q_{k},x\right) \right) \phi \leq \frac{C}{%
t_{k}-c_{0}}\sup_{\Omega }\phi +\int_{M}\left( \frac{1}{2}\left(
t_{k}-d\left( q_{k},x\right) \right) -a\right) \phi   \label{w10} \\
+\varepsilon \sup_{\Omega }\phi ,  \notag
\end{gather}%
where we have used 
\begin{gather*}
\int_{\mathbb{S}^{n-1}}J\left( q_{k},r_{k},\xi _{k}\right) \ln J\left(
q_{k},r_{k},\xi _{k}\right) \phi d\xi _{k}\geq -\frac{1}{e}\int_{\mathbb{S}%
^{n-1}}\phi d\xi _{k} \\
\geq -C\sup_{\Omega }\phi .
\end{gather*}%
We now define 
\begin{equation*}
\beta _{k}\left( x\right) :=t_{k}-d\left( q_{k},x\right) .
\end{equation*}%
Clearly, $\beta _{k}\left( x_{0}\right) =0$ and $\left\vert \nabla \beta
_{k}\right\vert \leq 1.$ Hence, a subsequence of $\beta _{k}$ converges, as $%
k\rightarrow \infty ,$ uniformly on compact sets to a Lipschitz function $%
\beta $ satisfying 
\begin{equation*}
\left\vert \beta \right\vert \left( x\right) \leq d\left( x_{0},x\right) 
\text{ on }M.
\end{equation*}%
Moreover, taking limit in (\ref{w10}) implies that in the weak sense 
\begin{equation}
\Delta _{f}\beta \geq a-\frac{1}{2}\beta \text{ \ \ on }M.  \label{w11}
\end{equation}%
Let $\phi $ be a cut-off function with $\phi =1$ on $B_{x_{0}}\left( \frac{1%
}{2}\right) ,$ $\phi =0$ on $M\backslash B_{x_{0}}\left( 1\right) $ and $%
\left\vert \nabla \phi \right\vert \leq 2.$ Multiplying (\ref{w11}) by $\phi 
$ and integrating, we get 
\begin{eqnarray*}
a\mathrm{Vol}_{f}\left( B_{x_{0}}\left( \frac{1}{2}\right) \right)  &\leq &%
\frac{1}{2}\int_{M}\phi \beta e^{-f}+\int_{M}\left( \Delta _{f}\beta \right)
\phi e^{-f} \\
&\leq &\frac{1}{2}\int_{B_{x_{0}}\left( 1\right) }d\left( x_{0},\cdot
\right) e^{-f}-\int_{M}\left\langle \nabla \beta ,\nabla \phi \right\rangle
e^{-f} \\
&\leq &3\mathrm{Vol}_{f}\left( B_{x_{0}}\left( 1\right) \right) .
\end{eqnarray*}%
This means that 
\begin{equation}
a\leq 3\frac{\mathrm{Vol}_{f}\left( B_{x_{0}}\left( 1\right) \right) }{%
\mathrm{Vol}_{f}\left( B_{x_{0}}\left( \frac{1}{2}\right) \right) }.
\label{a}
\end{equation}%
On the other hand, the Laplacian comparison theorem in \cite{WW} implies
that for $\frac{1}{4}\leq r\leq 1,$ 
\begin{equation*}
\Delta _{f}r\leq \ c\left( n\right) \left( 1+\sup_{B_{x_{0}}\left( 1\right)
}\left\vert f\right\vert \right) .
\end{equation*}%
Therefore, after integrating, we have 
\begin{equation*}
\frac{\mathrm{Vol}_{f}\left( B_{x_{0}}\left( 1\right) \right) }{\mathrm{Vol}%
_{f}\left( B_{x_{0}}\left( \frac{1}{2}\right) \right) }\leq c\left( n\right)
\exp \left( 16\sup_{B_{x_{0}}\left( 1\right) }\left\vert f\right\vert
\right) .
\end{equation*}%
Now the theorem follows by combining this with (\ref{a}).
\end{proof}

\newpage

\address {\noindent Department of Mathematics\\ University of Connecticut\\
Storrs, CT 06268\\ USA\\ \email{\textit{E-mail address}: {\tt
ovidiu.munteanu@uconn.edu} } \vskip 0.3in \address {\noindent School of
Mathematics \\ University of Minnesota\\ Minneapolis, MN 55455\\ USA\\ 
\email{\textit{E-mail address}: {\tt jiaping@math.umn.edu}} \end{document}